\title{Solvability by power of sum of element order}
\author{hiranya.dey }

\documentclass[preprint, 11pt]{elsarticle}

\usepackage{amssymb}
\usepackage{amsthm}
\usepackage{graphics}
\usepackage{epsfig}
\usepackage{amssymb}
\usepackage{graphicx}
\usepackage{subfigure}
\usepackage{color}
\usepackage{tikz}
\usepackage{tikz,pgfplots}
\pgfplotsset{compat=1.12,axis lines=center}
\usepgfplotslibrary{fillbetween}
\usetikzlibrary{patterns}
\usepackage{hyperref}
\usepackage{a4wide}
\usepackage{amsmath}
\usepackage{amsfonts}
\usepackage{enumerate}
\usepackage{multicol}

\makeatletter
 \def\ps@pprintTitle{%
   \let\@oddhead\@empty
  \let\@evenhead\@empty
  \let\@oddfoot\@empty
  \let\@evenfoot\@oddfoot
 }
\makeatother

\hypersetup{
    colorlinks=true,
    linkcolor=blue,
    filecolor=dviolet,      
    urlcolor=blue,
} 

\newtheorem*{theoremaux}{Theorem \theoremauxnum}
\gdef\theoremauxnum{1}

\newtheorem{lemma}{\bf Lemma}[section]
 
\newtheorem{theorem}{\bf Theorem}[section]

\newtheorem{remark}{\bf Remark}[section]

\journal{~}

\begin{document}

\begin{frontmatter}



\title{The solvability of a finite group by the sum of powers of element orders}



\author{Hiranya Kishore Dey}
\ead{hiranya.dey@gmail.com} 
\address{Department of Mathematics,\\ 
Indian Institute of Science, Bangalore \\ India}

%


\begin{abstract}

\noindent 
We prove a new criterion for the solvability of the finite groups, depending on the function $\psi_k(G)$ which is defined as the sum of $k$-th powers of the element orders of $G$. We show that our result can be used to show the solvability of some groups for which the solvability does not follow from earlier similar kind of results and we emphasize the following: looking at $\psi_k(G)$ for $k>1$ can be useful to get further pieces of information about the group $G$.

\medskip

\noindent

\end{abstract}


\end{frontmatter}

\section{Introduction}
\label{sec:intro}

Let $G$ be a finite group. H. Amiri, S. M. Jafarian Amiri and I. M. Isaacs  \cite{amiri-communication} defined the following function:
$$\psi(G)= \sum_{g \in G} o(g),$$ 
where $o(g)$ denotes the order of the element $g$.  They proved that
for any finite group $G$ of order $n$, $\psi(G) \leq \psi(\mathbb{Z}_n)$ and equality holds if and only if $G \simeq \mathbb{Z}_n,$ where $\mathbb{Z}_n$ denotes the cyclic group of order $n$.
This function $\psi$ has been considered in various works since then (see \cite{amiri-algapplctn, 
amiri-communication-secondmax, amiri-pureandApplied,
asad-joa, chew-chin-lim, Her-joa, Her-cia, Her-jpaa, shen-et-al, tarnauceanu-israel}) and the purpose of many of those papers was to prove new criteria for structural properties like solvability, nilpotency, etc. of finite groups.

M. Herzog, P. Longobardi, and M. Maj  \cite{Her-joa} proved the following criterion for the solvability of a finite group. 

\begin{theorem}[Herzog-Longobardi-Maj]
\label{thm:solvable_bysumofel_herzog18JOA}
Let $G$ be a finite group of order $n$ and suppose that $$\psi(G) > \frac{1}{6.68}\psi(\mathbb{Z}_n).$$ Then $G$ is solvable. 
\end{theorem}

In the same paper \cite{Her-joa}, they mentioned that $\psi(A_5)=211$ and $\psi(\mathbb{Z}_{60})=1617.$ Thus $\psi(A_5) = \frac{211}{1617} \psi(\mathbb{Z}_{60})$  and that points towards the fact that Theorem \ref{thm:solvable_bysumofel_herzog18JOA} is not very far from the best possible. Regarding this, they conjectured that if $G$ is a group of order $n$ and
$\psi(G)> \frac{211}{1617} \psi(\mathbb{Z}_n)$, then $G$ is solvable. This conjecture has been settled by 
M. B. Azad and B. Khosravi \cite{asad-joa}.

\begin{theorem}[Azad-Khosravi]
\label{thm:solvable_bysumofel_herzog22JOA}
Let $G$ be a finite group of order $n$. If $$\psi(G) > \frac{211}{1617}\psi(\mathbb{Z}_n),$$ then $G$ is solvable. 
\end{theorem}

Recently E. I. Khukhro, A. Moret\'{o} and M. Zarrin \cite{Khukhro-Joa-2021} posed another conjecture regarding the solvability of a finite group, depending on the average of the order of the elements of a group. This was very recently proved by M. Herzog, P. Longobardi and M. Maj \cite{Her-joa-2022}.  

\begin{theorem}[Herzog-Longobardi-Maj]
\label{theorem:Khukro2021}
Let $G$ be a finite group of order $n$ satisfying $$\psi(G)< \frac{n\psi(A_5)}{60},$$ then $G$ is solvable. 
\end{theorem} 

M T\~{a}rn\~{a}uceanu \cite[Theorem 1.1]{tarnauceanu-israel} proved the following:

\begin{theorem}[T\~{a}rn\~{a}uceanu]
\label{theorem:;tarnauceanu}
Let $G$ be a finite group of order $n$. If $$\frac{\psi(G)}{n^2} > \frac{211}{3600},$$ then $G$ is solvable. 
\end{theorem} 

S. M. Jafarian Amiri and M. Amiri \cite{amiri-amiri-cia} considered the following generalization of the function $\psi(G)$ which is defined as follows: 
$$\psi_k(G)= \sum_{g \in G} o(g)^k$$ 
for positive integers $k \geq 1$. They proved that for any positive integers $k$ and $n$ and any group $G$ of order $n$, $\psi_k(G) \leq \psi_k(\mathbb{Z}_n)$. 
Later, this function was also considered by S. Saha \cite{suvra}.
M. Herzog, P. Longobardi, and M. Maj \cite{Her-pureandApplied}  proved that if $G$ is a non-cyclic group of order $n,$ then $\psi(G) \leq \frac{7}{11} \psi(\mathbb{Z}_n).$  This result has been recently generalized by H. K. Dey and A. Mondal \cite{dey-archita-arxiv2022} 
and it was proved that if
$G$ is a non-cyclic group of order $n$ and $k$ is any fixed positive integer, then
$\psi_k(G) \leq \frac{1+3.2^k}{1+2^k+2.4^k} \psi_k(\mathbb{Z}_n).$

In this paper, we are interested in the solvability of a finite group, depending on the value of $\psi_k$. In this context, the main result of this paper is the following.

\begin{theorem}
\label{thm:solvability_criterion}
Let $G$ be a finite group of order $n$ such that the largest prime divisor of $n$ is $p$.
\begin{enumerate}
\item For $p>7$, if there exists any positive integer $k\geq 4$ such that 
$$\psi_k(G) > \displaystyle \frac{1+15.2^k+20.3^k+24.5^k}{(1+2^k+2.4^k)(1+2.3^k)(1+4.5^k)}\psi_k(\mathbb{Z}_n),$$
then the group $G$ is solvable.

\item For $p=7$, if there exists any positive integer $k\geq 13$ such that 
$$\psi_k(G) > \displaystyle \frac{1+15.2^k+20.3^k+24.5^k}{(1+2^k+2.4^k)(1+2.3^k)(1+4.5^k)}\psi_k(\mathbb{Z}_n),$$
then the group $G$ is solvable.
\end{enumerate}
\end{theorem}

 We also note that $\psi_k(A_5)= 1+15.2^k+20.3^k+24.5^k$ and $\psi_k(\mathbb{Z}_{60})= \psi_k(\mathbb{Z}_4) \times \psi_k(\mathbb{Z}_3) \times \psi_k(\mathbb{Z}_5)= (1+2^k+2.4^k)(1+2.3^k)(1+4.5^k)$ and $A_5$ is not solvable. Therefore if one wants to prove a similar result including $p=5$, this lower bound is the best possible.

Before going to the proof of Theorem \ref{thm:solvability_criterion}, we give examples of groups whose solvability does not follow from any of Theorem \ref{thm:solvable_bysumofel_herzog18JOA}, Theorem \ref{thm:solvable_bysumofel_herzog22JOA}, Theorem  \ref{theorem:Khukro2021}, and Theorem \ref{theorem:;tarnauceanu} but it follows from Theorem \ref{thm:solvability_criterion}. This also hints towards the fact that looking at $\psi_k(G)$ for different $k$, instead of looking only at $\psi(G)$ may be useful in getting further information about the structure of the group. 

There are $18$ groups up to isomorphism of order $156$. For a group $H$ of order $156$, by using Sage one can check the following:
\begin{enumerate}
\item If $\psi(H) > \frac{1}{6.68} \psi(\mathbb{Z}_{156}),$ then by Theorem \ref{thm:solvable_bysumofel_herzog18JOA}, $H$ is solvable. Thus, if $\psi(H)<1809$, Theorem \ref{thm:solvable_bysumofel_herzog18JOA} cannot be used to conclude the solvability of $H$. 
\item In a similar way, if $\psi(H)<1577$ then Theorem
\ref{thm:solvable_bysumofel_herzog22JOA} cannot be used. 
\item We have $\psi(A_5)=211.$  Thus, by Theorem \ref{theorem:Khukro2021}, if $\psi(H)\leq548$, then $H$ is solvable. Thus, if $\psi(H)>549,$ Theorem \ref{theorem:Khukro2021}
cannot be used. 
\item If $\psi(H)<1426,$ then 
Theorem \ref{theorem:;tarnauceanu} cannot be used to conclude the solvability of $H$.
\end{enumerate}

Thus for the groups $H$ with $549<\psi(H)<1426,$ one cannot deduce the solvability of $H$ by any of Theorem \ref{thm:solvable_bysumofel_herzog18JOA}, Theorem \ref{thm:solvable_bysumofel_herzog22JOA}, Theorem  \ref{theorem:Khukro2021}, and Theorem \ref{theorem:;tarnauceanu}. The group $H_1=\mathbb{Z}_2 \times (\mathbb{Z}_{13} : \mathbb{Z}_6)$ (Gap group id (156, 8)) and
$H_2= \mathbb{Z}_2 \times \mathbb{Z}_2 :  (\mathbb{Z}_{13} : \mathbb{Z}_3)$ (Gap group id (156, 14)) are two such groups. Using Sage, we have checked that both of these groups satisfy the condition of Theorem \ref{thm:solvability_criterion} for $k=4$ and thus, both of them are solvable. 


Throughout this article, $\phi$ is the Euler's function. For a group $G$ and a subgroup $P$ of $G$, by $N_G(P)$, we denote the Normalizer of the subgroup $P$ in the group $G$. Most of our notation is standard and we refer the reader to the books \cite{Isac-ams, Scott}. 

\section{Proof of Theorem \ref{thm:solvability_criterion}}

In this section, we at first recall some earlier known results
which are crucial for the proof of our main result. M. Herzog, P. Longobardi and M. Maj \cite[Theorem 10]{Her-jpaa} proved the following.

\begin{theorem}
\label{thm:solv-eulerphi}
Let $G$ be a finite group of order $n$ satisfying $$\psi(G) \geq \frac{3n}{5} \phi(n).$$
Then $G$ is solvable. 
\end{theorem}

S. M. Jafarian Amiri and H. Amiri \cite[Lemma 2.3, Lemma 2.5]{amiri-amiri-cia} proved the following two lemmas.

\begin{lemma}
\label{lem:analogue_cia} 
Let $P$ be a Sylow $p$-subgroup of $G$, $P \trianglelefteq G$ and $P$ is cyclic in $G$. Then, 
$$\psi_k(G) \leq \psi_k(P) \psi_k(G/P).$$ 
Moreover, equality happens if and only if $P \subseteq Z(G).$
\end{lemma}

\begin{lemma}
\label{lem:coprime}
For any fixed positive integer $k$ and any two finite groups $A$ and $B$, we have 
$\psi_k(A \times B) \leq \psi_k(A) \times \psi_k(B).$ Moreover, equality holds if and only if 
 $\gcd(|A|, |B|)=1$. 
\end{lemma}

The proof of Lemma \ref{lem:analogue_cia} can also be found in \cite{dey-archita-arxiv2022}.  
We also mention the following result of Herstein \cite{herstein}.

\begin{theorem}
\label{thm:herstein} 
Let $G$ be a finite group with an abelian maximal subgroup $C$. Then $G$ is solvable.  
\end{theorem}

We now prove our main result.

\begin{proof}[Proof of Theorem \ref{thm:solvability_criterion}]
We first prove the first part of the theorem. For that, 
we prove the following claim:\\
{\bf Claim:} For $p>7$ and $k \geq 4$, we have 
\begin{equation}
\label{eqn:main} \frac{1+15.2^k+20.3^k+24.5^k}{(1+2^k+2.4^k)(1+2.3^k)(1+4.5^k)}> \frac{1}{2^kp^{k-1}}.
\end{equation}

\noindent
{\bf Proof of Claim:} For $k\geq 4,$ we have $p^{k-1}\geq p^{k-4}p^3\geq 6^{k-4}6^4 =6^k.$ Thus, after cross-multiplying, it is sufficient to show the following:
\begin{eqnarray*}
12^k+15.24^k+20.36^k+24.60^k & \geq &  1+4.5^k+2.3^k+8.15^k+2^k+4.10^k+2.6^k+8.30^k \\ & & +2.4^k+8.20^k+4.12^k+16.60^k.
\end{eqnarray*}  
This directly follows by rearranging the terms of the right hand side suitably. This proves the claim.

We now proceed towards the proof of the theorem. Let us denote $\frac{1+15.2^k+20.3^k+24.5^k}{(1+2^k+2.4^k)(1+2.3^k)(1+4.5^k)}$ by $D_k$.
As there are at least $\phi(n)$  many elements of order $n$ in $\mathbb{Z}_n$ and it is easy to see that $\phi(n) \geq \frac{n}{p}$, by our assumptions and the above claim we have
$$\psi_k(G) > D_k \psi_k(\mathbb{Z}_n) > \frac{n^{k+1}}{2^kp^k}.$$
Therefore, there is at least one element $x$ such that $o(x)^k > \frac{n^k}{(2p)^k}$ and hence the order of $x$ is $>\frac{n}{2p}.$
Thus, $[G: \langle x \rangle ]< 2p$. We now proceed by induction on the number of prime divisors of $n$, say $l$.

If $l=1$, it is a $p$-group and hence solvable. Thus, let us assume that $l>1$ and the statement is true for $l-1$. 

There can be two cases.

\medskip

{\bf Case 1: $p$ divides $[G: \langle x \rangle ]$:} In this case, $[G: \langle x \rangle ]=p$ and thus $\langle x \rangle$ is a cyclic maximal subgroup of $G$ and hence by Theorem \ref{thm:herstein}, $G$ is solvable. 

\medskip

{\bf Case 2: $p$ does not divide $[G: \langle x \rangle ]$:} Thus $x$ contains a cyclic Sylow $p$-subgroup, say $P$ of $G$. 

If $P$ is normal, then using Lemma \ref{lem:analogue_cia}, we have
$$\psi_k(P)\psi_k(G/P) \geq \psi_k(G) \geq D_k \psi_k (\mathbb{Z}_n) = D_k\psi_k(\mathbb{Z}_{|P|}) \psi_k(\mathbb{Z}_{|G/P|}).
$$ The last equality follows from Lemma \ref{lem:coprime}. Since
$\psi_k(P)= \psi_k(\mathbb{Z}_{|P|})$, we have $$\psi_k(G/P) \geq D_k\psi_k(\mathbb{Z}_{|G/P|}).$$
As the number of prime divisors of $G/P$ must be $l-1$, by induction $G/P$ is solvable and thus $G$ is solvable.

If $P$ is not normal in $G$, as $\langle x \rangle$ is a subgroup of $N_G(P)$, it follows that $[G:N_G(P)] <2p$. Using Sylow Theorem, we have that $[G:N_G(P)] =p+1.$ As the index of $N_G(P)$ in $G$ is $p+1$ and the index of $\langle x \rangle $in $G$ is $2p$, and  $\langle x \rangle$ is a subgroup of $N_G(P)$, we must have that $\langle x \rangle = N_G(P).$ Thus, it is again a cyclic maximal subgroup of $G$ and hence $G$ is solvable. 

This completes the proof of the first part of the theorem.

For the second part, we observe that $7^{k-1}>6^k$ for $k \geq 13$ and using this, one can prove that for $p=7$ and $k\geq 13$, we have $$\frac{1+15.2^k+20.3^k+24.5^k}{(1+2^k+2.4^k)(1+2.3^k)(1+4.5^k)}\geq \frac{1}{2^k7^{k-1}}$$ 
in an identical way to the proof of \eqref{eqn:main}. The rest of the proof for the second part follows exactly as the first part.
This completes the proof of the theorem. 
\end{proof}

\begin{remark}
By Burnside's Theorem for solvable groups, we know that any group of order $p^aq^b$ is solvable. Thus, Theorem \ref{thm:solvability_criterion}, together with Burnside's Theorem, gives a sufficient condition of solvability for any group not of the form $2^a3^b5^c.$
\end{remark}

We end this short note with a straightforward generalization of Theorem \ref{thm:solv-eulerphi}.
For that purpose, we at first prove a relation between $\psi_k(G)$ and $\psi(G)$.

\begin{lemma}
\label{lem::reltn_psik-psiG}
Let $G$ be a non-cyclic group of order $n$ such that the smallest prime divisor of $n$ is $q$. Then, $\psi_k(G) \leq \big(\frac{n}{q}\big)^{k-1}\psi(G)$.
\end{lemma}

\begin{proof}
We have
\begin{equation*}
\psi_k(G)=\sum _{a \in G} o(a)^k \leq  \sum _{a \in G} o(a).\bigg(\frac{n}{q}\bigg)^{k-1}=\bigg(\frac{n}{q}\bigg)^{k-1}\psi(G).
\end{equation*}
The second inequality follows since $q$ is the smallest prime divisor of $n$ and thus the maximum possible order of any element of the group can be $n/q.$ This completes the proof. 
\end{proof}

\begin{theorem}
\label{thm:solvability-reltn-phin} Let $G$ be a finite group of order $n$ such that the smallest prime divisor of $n$ is $q$ . If there exists a positive integer $k$ such that $$\psi_k(G)\geq \bigg(\frac{3n^k}{5q^{k-1}}\bigg) \phi(n),$$ then $G$ is solvable. 
\end{theorem}

\begin{proof} 
Follows from Lemma \ref{lem::reltn_psik-psiG} and Theorem \ref{thm:solv-eulerphi}. 
\end{proof}

 \section*{Acknowledgements}
 
 \noindent 
The author is thankful to Prof.
 Arvind Ayyer for his constant support and encouragement. The author also acknowledges Prof. Angsuman Das and Archita Mondal for helpful discussions. 
 The author acknowledges SERB-National Post Doctoral Fellowship (File No. PDF/2021/001899) during the preparation of this work and profusely thanks Science and Engineering Research Board, Govt. of India for the funding. The author also acknowledges ideal working conditions in the Department of Mathematics, Indian Institute of Science.


\medskip




\begin{thebibliography}{20}





\bibitem{amiri-algapplctn} H. Amiri, S. M. Jafarian Amiri. \emph{Sum of element orders on finite groups of the same order},
Journal of Algebra and its Applications \textbf{10} (2011), 187-190.

\bibitem{amiri-communication} H. Amiri, S. M. Jafarian Amiri, I. M. Isaacs. \emph{Sums of element orders in finite groups}, Communications in Algebra \textbf{37} (2009), 2978-2980. 

\bibitem{amiri-communication-secondmax} S. M. Jafarian Amiri. \emph{Second maximum sum of element orders on finite nilpotent groups}, Communications in Algebra \textbf{41} (2013), 2055-2059.




\bibitem{amiri-pureandApplied} S. M. Jafarian Amiri, M. Amiri. \textit{Second maximum sum of element orders on finite groups},
Journal of Pure and Applied Algebra \textbf{218} (2014), 531-539.

\bibitem{amiri-amiri-cia} S. M. Jafarian Amiri, M. Amiri. \textit{Sum of the Products of the Orders of Two Distinct
Elements in Finite Groups},
Communications in Algebra \textbf{42} (2014), 5319-5328.



\bibitem{asad-joa} M. Baniasad Asad and B. Khosravi, \emph{A criterion for solvability of a finite group by the sum of element orders,} Journal of Algebra \textbf{516} (2018), 115-124.

\bibitem{chew-chin-lim} C. Y. Chew, A. Y. M. Chin, and C. S. Lim, \emph{A Recursive Formula for the Sum of Element Orders of Finite Abelian Groups,} Results in Mathematics \textbf{72} (2017), 1897-1905. 

\bibitem{dey-archita-arxiv2022} 
H. K. Dey, and A. Mondal, \emph{An exact upper bound for the sum of powers of element orders in non-cyclic finite groups,
} avaiable at arXiv, 
https://doi.org/10.48550/arXiv.2208.05161, 2022. 

\bibitem{herstein}
I. N. Herstein, \emph{A remark on finite groups}, Proceedings of the American Mathematical Society \textbf{9} 1958, 255-257. 


\bibitem{Her-pureandApplied} M. Herzog, P. Longobardi and M. Maj, \emph{An exact upper bound for sums of element orders in non-cyclic finite groups,} Journal of Pure and Applied Algebra \textbf{222} (2018), 1628-1642.

\bibitem{Her-joa}M. Herzog, P. Longobardi and M. Maj, \emph{Two new criteria for solvability of finite groups in finite groups,} Journal of Algebra \textbf{511} (2018), 215-226.

\bibitem{Her-cia}M. Herzog, P. Longobardi and M. Maj, \emph{Sums of element orders in groups of order 2m with m odd,} Communications in Algebra \textbf{47} (2019), 2035-2948.

\bibitem{Her-jpaa}M. Herzog, P. Longobardi and M. Maj, \emph{The second maximal groups with respect to the sum of element orders,} Journal of Pure and Applied Algebra \textbf{225} (2021), 106531.

\bibitem{Her-joa-2022}M. Herzog, P. Longobardi and M. Maj, \emph{Another criterion for solvability of finite groups,} Journal of Algebra \textbf{597} (2022), 1-23.


\bibitem{Isac-ams}I. M. Isaacs, \emph{Finite Group Theory,} Graduate Studies in mathematics, Vol. 92, American Mathematical Society, Providence, RI, 2008.

\bibitem{Khukhro-Joa-2021} E. I. Khukhro, A. Moret\'{o}, and M. Zarrin,
\emph{The average element order and the number of conjugacy classes of finite groups}, Journal of Algebra, \textbf{569} (2021), 1-11.

\bibitem{suvra} S. Saha, \emph{Sum of the Powers of the Orders of Elements in Finite Abelian Groups}, Advances in Group Theory and Applications, \textbf{13} (2022), 1-11.

\bibitem{Scott}W. R. Scott. \emph{Group Theory,} Prentice-Hall, Englewood Cliffs, NJ, 1964.


\bibitem{shen-et-al} R. Shen, G. Chen, C. Wu. \textit{On groups with the second largest value of the sum of element orders},
Communications in Algebra \textbf{43} (2015), 2618-2631.

\bibitem{tarnauceanu-israel} M. T\~{a}rn\~{a}uceanu, \textit{Detecting structural properties of finite groups by the sum of element orders}, Israel Journal of Mathmatics \textbf{238} (2020), 629–637. 


\end{thebibliography}
\end{document}